\documentclass[12pt,oneside,english,jou]{amsart}
\usepackage[T1]{fontenc}
\usepackage{babel}
\usepackage{verbatim}
\usepackage{mathrsfs}
\usepackage{amstext}
\usepackage{amsthm}
\usepackage{amssymb}
\usepackage{amsfonts}
\usepackage{amsmath}
\usepackage{esint}
\usepackage{enumerate}
\usepackage[unicode=true,pdfusetitle,
bookmarks=true,bookmarksnumbered=false,bookmarksopen=false,
breaklinks=false,pdfborder={0 0 1},colorlinks=false]
{hyperref}
\usepackage[margin=2.5cm]{geometry}
\usepackage[foot]{amsaddr}
\usepackage{mathtools}
\usepackage{ dsfont }

\makeatletter
\theoremstyle{plain}
\newtheorem{thm}{\protect\theoremname}[section]
\theoremstyle{definition}
\newtheorem{rem}[thm]{\protect\remarkname}
\theoremstyle{definition}
\newtheorem{defn}[thm]{\protect\definitionname}
\theoremstyle{plain}

\theoremstyle{plain}
\newtheorem{lem}[thm]{\protect\lemmaname}
\theoremstyle{plain}

\theoremstyle{plain}
\newtheorem{cor}[thm]{\protect\corollaryname}

\theoremstyle{definition}

\newtheorem*{ack}{Acknowledgement}
\theoremstyle{definition}

\theoremstyle{definition}

\theoremstyle{definition}

\usepackage{tikz}
\usetikzlibrary{shapes,arrows}
\usepackage{verbatim}
\usepackage{amsthm}
\usepackage{amstext}

\DeclareMathOperator{\Ber}{Ber}

\DeclareMathOperator{\supp}{supp}

\newcommand{\R}{\mathbb R}

\newcommand{\N}{\mathbb N}
\newcommand{\Q}{\mathbb Q}

\newcommand{\PP}{\mathbb P}

\newcommand{\FF}{\mathcal F}
\newcommand{\II}{\mathcal I}

\newcommand{\ii}{\underline{i}}

\newcommand{\hdim}{\dim_H}

\makeatother

\providecommand{\conjecturename}{Conjecture}
\providecommand{\corollaryname}{Corollary}
\providecommand{\definitionname}{Definition}
\providecommand{\examplename}{Example}
\providecommand{\lemmaname}{Lemma}
\providecommand{\problemname}{Problem}
\providecommand{\propositionname}{Proposition}
\providecommand{\remarkname}{Remark}
\providecommand{\theoremname}{Theorem}
\providecommand{\taskname}{Task}

\begin{document}

\title{On the dimension of stationary measures for random piecewise affine interval homeomorphisms}

\author[K. Bara\'{n}ski]{Krzysztof Bara\'{n}ski$^1$}
\address{$^1$Institute of Mathematics, University of Warsaw, ul. Banacha 2, 02-097 Warszawa, Poland.}
\email{baranski@mimuw.edu.pl}

\author[A. \'{S}piewak]{Adam \'{S}piewak$^{2,3}$}
\address{$^2$Department of Mathematics,	Bar-Ilan University, Ramat-Gan, 5290002, Israel\ \ \
	$^3$Institute of Mathematics, Polish Academy of Sciences, ul.~\'Sniadeckich 8, 00-656 Warszawa, Poland}
\email{ad.spiewak@gmail.com}



\maketitle

\begin{abstract}
We study stationary measures for iterated function systems (considered as random dynamical systems) consisting of two piecewise affine interval homeomorphisms, called Alsed\`a--Misiurewicz (AM) systems. We prove that for an open set of parameters, the unique non-atomic stationary measure for an AM-system has Hausdorff dimension strictly smaller than $1$. In particular, we obtain singularity of these measures, answering partially a question of Alsed\`a and Misiurewicz from 2014.
\end{abstract}

\section{Introduction}
In recent years, a growing interest in low-dimensional random dynamics has led to an intensive study of random one-dimensional systems given by (semi)groups of interval and circle homeomorphisms, both from stochastic and geometric point of view (see e.g. \cite{alseda-misiurewicz,Czudek-Szarek,CzudekSzarekWojewodka,gelfert,homburg16,homburg,LuczynskaUniqueErgodicity,LuczynskaSzarek22,malicet,SZ1,SZ2}). This can be seen as an extension of the research on the well-known case of groups of smooth circle diffeomorphisms (see e.g.~\cite{ghys-survey,navas-book}).

Let $f_1, \ldots, f_m$, $m \geq 2$, be homeomorphisms of a $1$-dimensional compact manifold $X$ (a closed interval or a circle). The transformations $f_i$ generate a semigroup consisting of iterates $f_{i_n} \circ \cdots \circ f_{i_1}$, where $i_1, \ldots, i_n \in \{1, \dots, m\}$, $n \in \{0, 1, 2,\ldots\}$. 
For a probability vector  $(p_1, \ldots, p_m)$, such a system defines a Markov process on $X$ which, by the Krylov–Bogolyubov theorem, admits a (non-necessarily unique) \emph{stationary measure}, i.e.~a Borel probability measure $\mu$ on $X$ satisfying
\[
\mu(A) = \sum_{i = 1}^m p_i \mu(f_i^{-1}(A))
\]
for every Borel set $A \subset X$. In many cases it can be shown that the stationary measure is unique (at least within some class of measures) and is either absolutely continuous or singular with respect to the Lebesgue measure. It is usually a non-trivial problem to determine which of the two cases occurs (see e.g. \cite[Section 7]{navas}), and the question has been solved only in some particular cases.

This paper is a continuation of the research started in \cite{BS19Singular} on singular stationary measures for so-called Alsed\`a--Misiurewicz systems (AM-systems), defined in \cite{alseda-misiurewicz}. These are random systems generated by two piecewise affine increasing homeomorphisms $f_-$, $f_+$ of the unit interval $[0,1]$, such that $f_i(0) = 0$, $f_i(1) = 1$ for $i = -, +$, each $f_i$ has exactly one point of non-differentiability $x_i \in (0,1)$, and $f_-(x) < x < f_+(x)$ for $x \in (0,1)$. For a detailed description of AM-systems refer to \cite{BS19Singular}. The dynamics of AM-systems and related ones has already gained some interest in recent years, being studied in e.g.~\cite{alseda-misiurewicz,BS19Singular,Bradik-Roth,Czernous20,CzernousSzarek20,Czudek,Toyokawa}. Within the class of uniformly contracting iterated function systems, piecewise linear maps and the dimension of their attractors were recently studied in \cite{ProkajSimonPiecewise}.

In this paper, as explained below, we study stationary measures for symmetric AM-systems with positive endpoint Lyapunov exponents.

\begin{defn} \label{defn:AM}
A symmetric AM-system is the system $\{f_-, f_+\}$ of increasing homeomorphisms of the interval $[0,1]$ of the form
\[
f_-(x)=
\begin{cases}
a x &\text{for }x\in[0,x_-]\\
1 - b(1 - x) &\text{for }x\in (x_-, 1]
\end{cases}, \qquad
f_+(x)=
\begin{cases}
b x &\text{for }x\in[0,x_+]\\
1 - a(1 - x) &\text{for }x\in (x_+, 1]
\end{cases},
\]
where $0 < a < 1 < b$ and
\[
x_- = \frac{b - 1}{b - a}, \qquad x_+ = \frac{1 - a}{b - a}.
\]
See Figure~\ref{fig:graph}. 
\end{defn}
\begin{figure}[ht!]
\begin{center}
\includegraphics[width=0.45\textwidth]{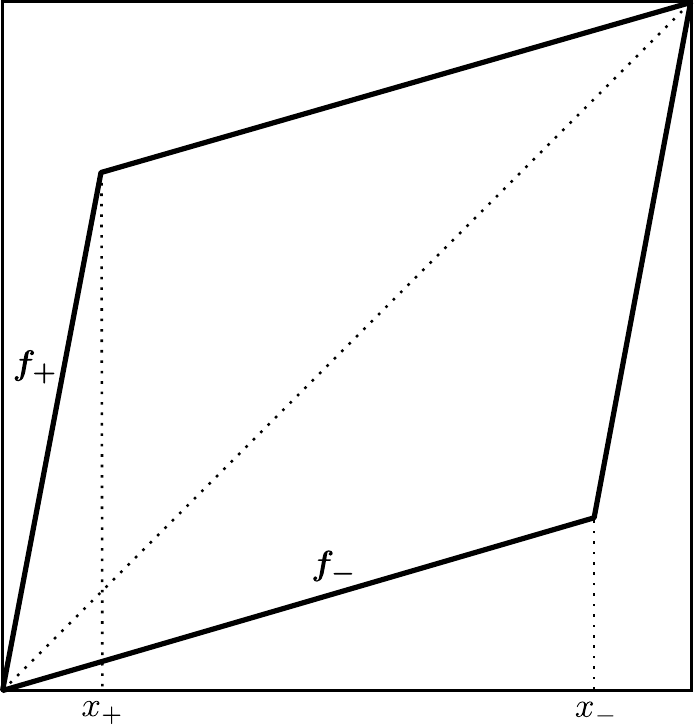}
\end{center}
\caption{An example of a symmetric AM-system.}\label{fig:graph}
\end{figure}
We consider $\{f_-, f_+\}$ as a random dynamical system, which means that iterating the maps, we choose them independently with probabilities $p_-, p_+$, where $(p_-, p_+)$ is a given probability vector (i.e.~$p_-, p_+ > 0$, $p_- + p_+ = 1$). Formally, this defines the \emph{step skew product}
	\begin{equation}\label{eq:skew}
	\FF^+\colon \Sigma_2^+ \times [0,1] \to \Sigma_2^+ \times [0,1], \qquad \FF^+(\ii, x) = (\sigma(\ii), f_{i_1}(x)),
	\end{equation}
where $\Sigma_2^+ = \{-,+\}^\N,\ \ii = (i_1, i_2, \ldots) \in \Sigma_2^+$ and $\sigma\colon \Sigma_2^+ \to \Sigma_2^+$ is the left-side shift. 

The \emph{endpoint Lyapunov exponents} of an AM-system $\{f_-, f_+\}$ are defined as
\[
\Lambda(0) = p_- \log f_-'(0) + p_+ \log f_+'(0),\qquad
\Lambda(1) = p_- \log f_-'(1) + p_+ \log f_+'(1).
\]
It is known (see \cite{alseda-misiurewicz,homburg}) that if the endpoint Lyapunov exponents are both positive, then the AM-system exhibits the \emph{synchronization} property, i.e.~for almost all paths $(i_1, i_2, \ldots) \in \{-, +\}^\N$ (with respect to the $(p_-, p_+)$-Bernoulli measure) we have $|f_{i_n} \circ \cdots \circ f_{i_1}(x) - f_{i_n} \circ \cdots \circ f_{i_1}(y)| \to 0$ as $n \to \infty$ for every $x,y \in [0,1]$. Moreover, in this case there exists a unique stationary measure $\mu$ without atoms at the endpoints of $[0,1]$, i.e.~a Borel probability measure $\mu$ on $[0,1]$, such that
\[
\mu = p_- \, (f_-)_* \mu + p_+ \, (f_+)_* \mu,
\]
with $\mu(\{0,1\}) = 0$ (see \cite{alseda-misiurewicz}, \cite[Proposition 4.1]{homburg16}, \cite[Lemmas~3.2--3.4]{homburg} and, for a more general case, \cite[Theorem 1]{Czudek-Szarek}). From now on, by a stationary measure for an AM-system we will always mean the measure $\mu$.
It is known that $\mu$ is non-atomic and is either absolutely continuous or singular with respect to the Lebesgue measure (see \cite[Propositions 3.10--3.11]{BS19Singular}). 

In \cite{alseda-misiurewicz}, Alsed\`a and Misiurewicz conjectured that the stationary measure $\mu$ for an AM-system should be singular for typical parameters.
In our previous paper \cite{BS19Singular} we showed that there exist parameters $a,b, (p_-, p_+)$, for which $\mu$ is singular with Hausdorff dimension smaller than $1$ (see \cite[Theorems~2.10 and~2.12]{BS19Singular}). These examples can be found among AM-systems with \emph{resonant} parameters, i.e. the ones with $\frac{\log a}{\log b} \in \Q$. In most of the examples the measure $\mu$ is supported on an \emph{exceptional minimal set}, which is a Cantor set of dimension smaller than $1$ (although we also have found examples of singular stationary measures with the support equal to the unit interval, see \cite[Theorem 2.16]{BS19Singular}). 

In this paper, already announced in \cite{BS19Singular}, we make a subsequent step to answer the Alsed\`a and Misiurewicz question, showing that the stationary measure $\mu$ is singular for an open set of parameters $(a,b)$ and probability vectors $(p_-, p_+)$. In particular, we find non-resonant parameters (i.e.~those with $\frac{\log a}{\log b} \notin \Q$), for which the corresponding stationary measure is singular (note that non-resonant AM-systems necessarily have stationary measures with support equal to $[0,1]$, see \cite[Proposition~2.6]{BS19Singular}). To prove the result, we present another method to verify singularity of stationary measures for AM-systems. Namely, instead of constructing a measure supported on a set of small dimension, we use the well-known bound on the dimension of stationary measure 
\[ 
\dim_H \mu \leq -\frac{H(p_-, p_+)}{\chi(\mu)},
\]
in terms of its \emph{entropy} 
\[H(p_-, p_+) = -p_-\log p_- -  p_+\log p_+\]
and the \emph{Lyapunov exponent}
\[ \chi(\mu) = \int \limits_{[0,1]} ( p_- \log f'_-(x) + p_+ \log f'_+(x) )d\mu(x), \]
proved in \cite{rams_jaroszewska} in a very general setting.
We find an open set of parameters for which the Lyapunov exponent is small enough (hence the average contraction is strong enough) to guarantee $\hdim\mu < 1$. The upper bound on the Lyapunov exponent follows from estimates of the expected return time to the interval 
\[
M = [x_+, x_-].
\]

\begin{rem}
One should note that the question of Alsed\`a and Misiurewicz has been answered when considered within a much broader class of general random interval homeomorphisms with positive endpoint Lyapunov exponents \cite{Bradik-Roth,CzernousSzarek20} and minimal random homeomorphisms of the circle \cite{Czernous20}. More precisely, Czernous and Szarek considered in \cite{CzernousSzarek20} the closure $\overline{\mathcal{G}}$ of the space $\mathcal{G}$ of all random systems $((g_-, g_+), (p_-, p_+))$ of absolutely continuous, increasing homeomorphisms $g_-, g_+$ of $[0,1]$, taken with probabilities $p_-, p_+$, such that $g_-, g_+$ are $C^1$ in some fixed neighbourhoods of $0$ and $1$, have positive endpoint Lyapunov exponents and satisfy $g_-(x) < x < g_+(x)$ for $x \in (0,1)$. In \cite[Theorem 10]{CzernousSzarek20}, they proved that for a generic system in $\overline{\mathcal{G}}$ (in the Baire category sense under a natural topology), the unique non-atomic stationary measure is singular. This result was extended by Brad\'ik and Roth in \cite[Theorem 6.2]{Bradik-Roth}, where they allowed the functions to be only differentiable at $0,1$, and showed that in addition to being singular, the stationary measure has typically full support. Similar results were obtained by Czernous \cite{Czernous20} for minimal systems on the circle. However, as the finite-dimensional space of AM-systems is meagre as a subset of the spaces considered in \cite{Bradik-Roth,Czernous20,CzernousSzarek20}, these results give no information on the singularity of stationary measures for typical AM-systems.

\end{rem}

\begin{ack} Krzysztof Bara\'nski was supported by the National Science
	Centre, Poland, grant no 2018/31/B/ST1/02495. Adam \'Spiewak acknowledges support from the Israel Science Foundation, grant 911/19. A part of this work was done when the second author was visiting the Budapest University of Technology and Economics. We thank Bal\'azs B\'ar\'any, K\'aroly Simon and R. D\'aniel Prokaj for useful discussions and the staff of the Institute of Mathematics of the Budapest University of Technology and Economics for their hospitality. 
\end{ack}

\section{Results}

We adopt a convenient notation 
\[
b = a^{-\gamma}
\]
for $a \in (0,1)$, $\gamma > 0$ and
\[
\II\colon [0,1] \to [0,1], \qquad \II(x) = 1 - x,
\]
so that a symmetric AM-system has the form
\begin{equation}\label{eq:AM_new_param}
f_-(x)=
\begin{cases}
a x &\text{for }x\in[0,x_-]\\
\II(a^{-\gamma}\II(x)) &\text{for }x\in (x_-, 1]
\end{cases}, \qquad
f_+(x)=
\begin{cases}
a^{-\gamma} x &\text{for }x\in[0,x_+]\\
\II(a\II(x)) &\text{for }x\in (x_+, 1]
\end{cases},
\end{equation}
where
\[
x_- = \frac{a^{-\gamma} - 1}{a^{-\gamma} - a}, \qquad x_+ = \frac{1 - a}{a^{-\gamma} - a}. 
\]
By definition, we have
\begin{equation}\label{eq:sym}
f_\pm = \II \circ f_\mp \circ \II^{-1} = \II \circ f_\mp \circ \II.
\end{equation}
Under this notation, the endpoint Lyapunov exponents for the system \eqref{eq:AM_new_param} and a probability vector $(p_-, p_+)$ are given by 
\[
\Lambda(0) = (p_- - \gamma p_+)\log a, \qquad \Lambda(1) = (p_+ - \gamma p_-)\log a.
\]
Throughout the paper we assume that $\Lambda(0)$ and $\Lambda(1)$ are positive, which is equivalent to
\begin{equation}\label{eq:gamma}
\gamma > \max \Big(\frac{p_-}{p_+}, \frac{p_+}{p_-}\Big).
\end{equation}
In particular, we have $\gamma > 1$. Note that this implies
\begin{equation}\label{x+<x-}
x_+ < x_-.
\end{equation}
Indeed, if $\gamma > 1$, then the endpoint Lyapunov exponents for $p_- = p_+ = 1/2$ are positive, so \eqref{x+<x-} follows from \cite[Lemma 4.1]{BS19Singular}.

The aim of this paper is to prove the following theorem.

\begin{thm}\label{thm:singularity_lyap} Consider a space of symmetric AM-systems $\{f_-, f_+\}$ of the form \eqref{eq:AM_new_param} with positive endpoint Lyapunov exponents. Then there is a non-empty open set of parameters $(a,\gamma) \in (0,1) \times (1, \infty)$ and probability vectors $(p_-, p_+)$, such that the corresponding stationary measure $\mu$ for the system $\{f_-, f_+\}$ is singular with Hausdorff dimension smaller than $1$. More precisely, there exists $\delta > 0$ such that for every $(p_-, p_+)$ with $p_-, p_+ < \frac{1}{2} + \delta$ there is a non-empty open interval $J_{p_-,p_+} \subset (1,\frac 3 2)$, depending continuously on $(p_-, p_+)$, such that for $\gamma \in J_{p_-,p_+}$ and $a \in (0, a_{max})$ for some $a_{max} = a_{max}(\gamma) > 0$,depending continuously on $\gamma$, we have
\[
\dim_H \mu \leq \frac{p\log p + (1-p)\log (1-p)}{\Big(1 - \frac{(1+\gamma)p^2(p+\gamma)}{\gamma - p(1-p)}\Big) \log a} < 1,
\]
where $p = \max(p_-, p_+)$. 

In particular, in the case $(p_-, p_+) = (\frac{1}{2},\frac{1}{2})$ we have 
\[
\dim_H \mu \leq \frac{(1 - 4 \gamma) \log 2}{(\gamma - 1)(3/2 - \gamma) \log a} < 1
\]
for $\gamma \in (1,\frac{3}{2})$, $a \in \Big(0, 2^{\frac{1 - 4\gamma}{(\gamma - 1)(3/2 - \gamma)}}\Big)$.
\end{thm}

\begin{rem}\label{rem:range} The range of probability vectors $(p_-, p_+)$ for which we obtain the singularity of $\mu$ for a non-empty open set of parameters $a, \gamma$, is rather small. As the proof of Theorem~\ref{thm:singularity_lyap} shows, suitable conditions for the possible values of $p  = \max(p_-, p_+)$ are given by the inequalities \eqref{eq:gamma'} and \eqref{eq:p,gamma}. Solving them, we obtain $p \in [\frac{1}{2}, p_0)$, where $p_0 = 0.503507...$ is the smaller of the two real roots of the polynomial $p^6-2p^5+5p^4-6p^3-2p^2+1$. As $p$ varies from $\frac 1 2$ to $p_0$, the range of allowable parameters $\gamma$ shrinks from the interval $(1, \frac 3 2)$ to a singleton. For such values of $p$ and $\gamma$, the  measure $\mu$ is singular for sufficiently small $a > 0$. See Figure~\ref{fig:p-gamma}.
\end{rem}

\begin{figure}[ht!]
\begin{center}
\includegraphics[width=0.6\textwidth]{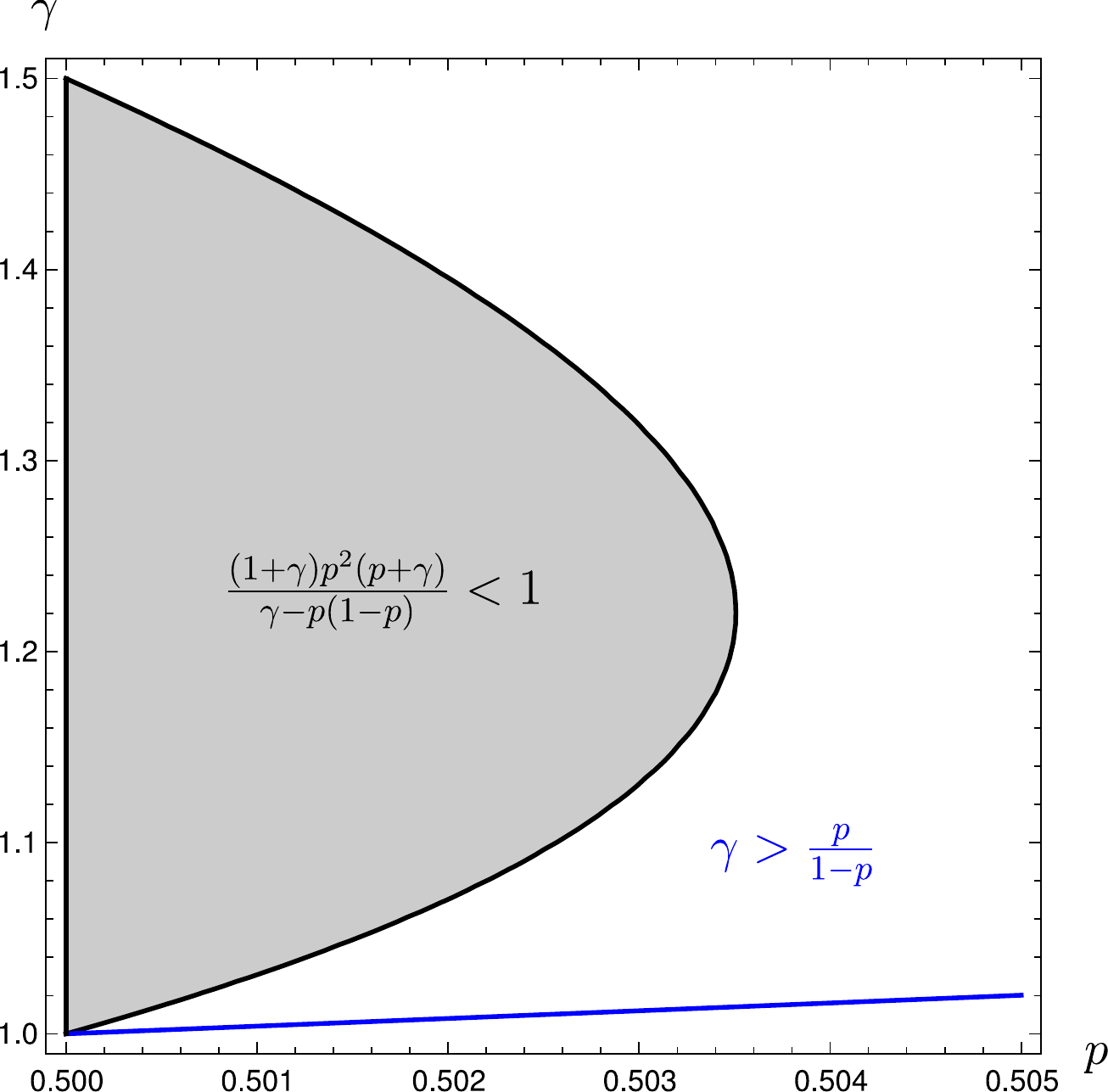}
\end{center}
\caption{The range of parameters $p = \max(p_-, p_+)$ and $\gamma$, for which the stationary measure $\mu$ for the system \eqref{eq:AM_new_param} is singular for sufficiently small $a > 0$.}\label{fig:p-gamma}
\end{figure}
\begin{rem}\label{rem:disjoint}
Every system of the form \eqref{eq:AM_new_param} with $a < \frac{1}{2}$ is of \emph{disjoint type} in the sense of \cite[Definition~2.3]{BS19Singular}, i.e.~$f_-(x_-) < f_+(x_+)$. Indeed, for $a < \frac{1}{2}$ we have
\[2a^{1-\gamma} < a^{-\gamma} < a^{-\gamma} +a, \]
so $a^{1-\gamma} -a < a^{-\gamma} - a^{1-\gamma}$ and
\[
f_-(x_-) = a \frac{a^{-\gamma} - 1}{a^{-\gamma} - a} <  a^{-\gamma} \frac{1 - a}{a^{-\gamma} - a} = f_+(x_+).
\]
Since a simple calculation shows $2^{\frac{1 - 4\gamma}{(\gamma - 1)(3/2 - \gamma)}} < \frac 1 2$ for $\gamma \in (1,\frac{3}{2})$, we see that all the systems with the probability vector $(p_-, p_+) = (\frac{1}{2},\frac{1}{2})$ covered by Theorem~\ref{thm:singularity_lyap} are of disjoint type.
\end{rem}

\begin{rem} Since the conditions used in the proof of Theorem~\ref{thm:singularity_lyap} to obtain the singularity of $\mu$ define open sets in the space of system parameters, it follows that the singularity of the stationary measure holds also for non-symmetric AM-systems with parameters close enough to the ones covered by Theorem~\ref{thm:singularity_lyap}. We leave the details to the reader.
\end{rem}

\section{Preliminaries}

We state some standard results from probability and ergodic theory, which we will use within the proofs. 

\begin{thm}[{\bf Hoeffding's inequality}{}]\label{thm:hoef}
Let $X_1, \ldots, X_n$ be independent bounded random variables and let $S_n = X_1 + \cdots + X_n$. Then for every $t > 0$,
\[
\PP(S_{n}-\mathbb {E} S_n\geq t)\leq \exp \left(-{\frac {2t^{2}}{\sum _{j=1}^{n}(\sup X_j - \inf X_j)^{2}}}\right)
\]
\end{thm}

\begin{thm}[{\bf Wald's identity}{}]\label{thm:wald}
Let $X_1, X_2, \ldots$ be independent identically distributed random variables with  finite expected value and let $N$ be a stopping time with $\mathbb{E} N < \infty$.  Then
\[
\mathbb{E} (X_1 + \cdots + X_N) = \mathbb{E} N \: \mathbb{E} X_1.
\]
\end{thm}

\begin{thm}[{\bf Kac's lemma}{}]\label{thm:kac}
Let $F\colon X \to X$ be a measurable $\mu$-invariant ergodic transformation of a probability space $(X, \mu)$ and let $A \subset X$ be a measurable set  with $\mu(A) > 0$. Then
\[
\int \limits_A n_A \, d\mu_A = \frac{1}{\mu(A)},
\]
where 
\[
n_A \colon X \to \N \cup \{\infty\}, \qquad n_A(x) = \inf\{n \ge 1: F^n(x) \in A\}
\]
is the first return time to $A$ and $\mu_A = \frac{1}{\mu(A)} \mu|_A$.
\end{thm}

For the proofs of these results refer to, respectively, \cite[Theorem 2]{Hoeffding_inequality}, \cite[Chapter XII, Theorem 2]{Feller}, \cite[Theorem 4.6]{petersen}.

\section{Proofs}

As noted in the introduction, the proof of Theorem \ref{thm:singularity_lyap} is based on an upper bound on the Hausdorff dimension of a stationary measure in terms of its entropy and Lyapunov exponent, in a version proved by Jaroszewska and Rams in \cite[Theorem 1]{rams_jaroszewska}. Consider a symmetric AM-system $\{f_-, f_+\}$ of the form \eqref{eq:AM_new_param} with positive endpoint Lyapunov exponents for some probability vector $(p_-,p_+)$, and its stationary measure $\mu$. Recall that the entropy of $(p_-,p_+)$ is defined by
\[H(p_-, p_+) = -p_-\log p_- -  p_+\log p_+,\]
while
\[ \chi(\mu) = \int \limits_{[0,1]} ( p_- \log f'_-(x) + p_+ \log f'_+(x) )d\mu(x) \]
is the Lyapunov exponent of $\mu$. As $\mu$ is non-atomic (see \cite[Proposition 3.11]{BS19Singular}) and $f_-, f_+$ are differentiable everywhere except for the points $x_-, x_+$, the Lyapunov exponent $\chi(\mu)$ is well-defined. Moreover, $\mu$ is ergodic (see \cite[Lemmas 3.2, 3.4]{homburg}). It follows that we can use \cite[Theorem 1]{rams_jaroszewska} which asserts that
\begin{equation}\label{eq:entropy_lyap} \dim_H \mu \leq -\frac{H(p_-, p_+)}{\chi(\mu)}
\end{equation}
as long as $\chi(\mu) < 0$.

Now we proceed with the details. Let
\[
M = [x_+, x_-], \qquad L = [x_+, f_-^{-1}(x_+)),\qquad  R=\II(L) = (f_+^{-1}(x_-), x_-].
\]
It follows from \eqref{x+<x-} that these intervals are well-defined. 
Note that $M, L, R$ depend on parameters $a$ and $\gamma$, but we suppress this dependence in the notation. To estimate the Hausdorff dimension of $\mu$, we find an upper bound for $\chi(\mu)$ in terms of $\mu(M)$ and estimate $\mu(M)$ from below. To this aim, we need the disjointness of the intervals $L,R$. The following lemma provides the range of parameters for which this condition holds.

\begin{lem}\label{lem:LR}
The following assertions are equivalent.
\begin{enumerate}[\rm (a)]
\item $\overline L \cap \overline R = \emptyset$,
\item $x_+ < f_-(\frac{1}{2})$,
\item $\gamma >1 - \frac{\log(a^2 - 2a +2)}{\log a}$.
\end{enumerate}
\end{lem}
\begin{proof} By \eqref{x+<x-} and the fact $x_+ = \II(x_-)$, we have $\frac{1}{2} < x_-$, so $f_-(\frac{1}{2}) = \frac{a}{2}$ and (b) becomes $x_+ < \frac{a}{2}$. Then a direct computation yields the equivalence of (b) and (c). Furthermore, by \eqref{x+<x-}, the condition (a) holds if and only if $f_-^{-1}(x_+) < f_+^{-1}(x_-)$. As $f_-\circ \II =  \II \circ f_+$, this is equivalent to $f_-^{-1}(x_+) < \II(f_-^{-1}(x_+))$, which is the same as $f_-^{-1}(x_+) < 1/2$. Applying $f_-$ to both sides, we arrive at (b).
\end{proof}

\begin{rem}\label{rem:LR}
The condition Lemma~\ref{lem:LR}(c) can be written as $\gamma - 1 > - \frac{\log((1-a)^2 + 1)}{\log a}$. As $\log((1-a)^2 + 1) < \log 2$ for $a \in (0, 1)$, we see that the condition is satisfied provided $\gamma > 1$, $a \in (0, 2^{\frac{1}{1 - \gamma}})$.
\end{rem}

We can now estimate the measure of $M$. It is convenient to use the notation
\[
p = \max(p_-, p_+).
\]
Obviously, $p \in [\frac 1 2, 1)$. Note that the condition \eqref{eq:gamma} for the positivity of the endpoint Lyapunov exponents can be written as
\begin{equation}\label{eq:gamma'}
\gamma > \frac{p}{1-p}
\end{equation}
and the entropy of $(p_-, p_+)$ is equal to
\[
H(p) = -p\log p -  (1-p)\log (1-p).
\]
The following lemma provides a lower bound for $\mu(M)$.

\begin{lem}\label{eq:middle_below}
	Let $a \in (0,1)$, $\gamma > 1$ and $p \in [\frac{1}{2}, 1)$ satisfy the conditions \eqref{eq:gamma'} and Lemma~{\rm \ref{lem:LR}(c)}.  Then
	\begin{equation*} \mu(M) \geq \frac{\gamma (1 - p) - p}{\gamma - p(1 - p)}.
	\end{equation*}
\end{lem}

Before giving the proof of Lemma~{\rm\ref{eq:middle_below}}, let us explain how it implies Theorem \ref{thm:singularity_lyap}. Suppose the lemma is true. Then we can estimate the Lyapunov exponent $\chi(\mu)$ in the following way.

\begin{cor}\label{cor:Lyap}
Let $a \in (0,1)$, $\gamma > 1$ and $p \in [\frac{1}{2}, 1)$ satisfy the conditions \eqref{eq:gamma'} and Lemma~{\rm \ref{lem:LR}(c)}. Then
\[
\chi(\mu) \le \Big(1 - \frac{(1+\gamma)p^2(p+\gamma)}{\gamma - p(1-p)}\Big) \log a.
\]
\end{cor}
\begin{proof}
By definition, we have
\begin{equation}\label{eq:chi formula}
\chi(\mu) =  (\mu (M) + (p_- - \gamma p_+)\mu([0, x_+]) + (p_+ - \gamma p_-) \mu([x_-, 1])) \log a.
\end{equation}
Computing the maximum of this expression under the condition $\mu([0, x_+]) + \mu([x_-, 1]) = 1 - \mu(M)$, we obtain
\[
\chi(\mu) \le (1 - (1+ \gamma) p (1 - \mu(M)))\log a.
\]
Then Lemma~{\rm\ref{eq:middle_below}} provides the required estimate by a direct computation.
\end{proof}

\begin{proof}[Proof of Theorem \rm \ref{thm:singularity_lyap}]
Let $a \in (0,1)$, $\gamma > 1$ and $p \in [\frac{1}{2}, 1)$ satisfy the conditions \eqref{eq:gamma'} and Lemma~{\rm \ref{lem:LR}(c)}. By Corollary~\ref{cor:Lyap}, we have $\chi(\mu) < 0$ provided 
\begin{equation}\label{eq:p,gamma}
\frac{(1+\gamma)p^2(p+\gamma)}{\gamma - p(1-p)} < 1.
\end{equation}
Hence, applying \eqref{eq:entropy_lyap} and Corollary~\ref{cor:Lyap}, we obtain
\begin{equation}\label{eq:dim<}
\dim_H \mu \leq \frac{p\log p + (1-p)\log (1-p)}{\Big(1 - \frac{(1+\gamma)p^2(p+\gamma)}{\gamma - p(1-p)}\Big) \log a}
\end{equation}
as long as \eqref{eq:p,gamma} is satisfied. If, additionally,
\begin{equation}\label{eq:p,gamma,a}
p\log p + (1-p)\log (1-p) > \Big(1 - \frac{(1+\gamma)p^2(p+\gamma)}{\gamma - p(1-p)}\Big) \log a,
\end{equation}
then \eqref{eq:dim<} provides $\dim_H \mu < 1$. We conclude that the conditions required for $\dim_H \mu < 1$ are \eqref{eq:gamma'}, \eqref{eq:p,gamma}, Lemma~{\rm \ref{lem:LR}(c)} and \eqref{eq:p,gamma,a}. 

To find the range of allowable parameters, consider first the case $p_- = p_+ = \frac 1 2$ (which corresponds to $p = \frac 1 2$). Then the condition \eqref{eq:gamma'} is equivalent to $\gamma > 1$, while the inequality \eqref{eq:p,gamma} takes the form $2\gamma^2 - 5 \gamma +3 < 0$ and is satisfied for $\gamma \in (1, \frac{3}{2})$. Furthermore, by Remark~\ref{rem:LR}, the condition Lemma~{\rm \ref{lem:LR}(c)} is fulfilled for $\gamma > 1$, $a \in (0, 2^{\frac{1}{1 -\gamma}})$. 
The condition \eqref{eq:p,gamma,a} can be written as $\frac{(1 - 4 \gamma) \log 2}{(\gamma - 1)(3/2 - \gamma) \log a} < 1$, which is  equivalent to
\[
a < 2^{\frac{1 - 4\gamma}{(\gamma - 1)(3/2 - \gamma)}}.
\]
A direct computation shows  $2^{\frac{1 - 4\gamma}{(\gamma - 1)(\frac{3}{2} - \gamma)}} < 2^{\frac{1}{1 - \gamma}}$ for $\gamma \in (1,\frac{3}{2})$. By \eqref{eq:dim<}, we conclude that in the case $p_- = p_+ = \frac 1 2$ we have
\[
\dim_H \mu \leq \frac{(1 - 4 \gamma) \log 2}{(\gamma - 1)(3/2 - \gamma) \log a} < 1
\]
for $\gamma \in (1,\frac{3}{2})$, $a \in \Big(0, 2^{\frac{1 - 4\gamma}{(\gamma - 1)(3/2 - \gamma)}}\Big)$.

Suppose now that $(p_-, p_+)$ is a probability vector with $p <  \frac 1 2 + \delta$ for a small $\delta > 0$. Note that the functions appearing in \eqref{eq:gamma'} and  \eqref{eq:p,gamma} are well-defined and continuous for $\gamma \in (1,\frac{3}{2})$ and $p$ in a neighbourhood of $\frac 1 2$. Hence, \eqref{eq:gamma'} and \eqref{eq:p,gamma} are fulfilled for $\gamma\in J_{p_-, p_+} = J_p$, where $J_p \subset (1, \frac{3}{2})$ is an interval slightly smaller than $(1, \frac{3}{2})$, depending continuously on $p \in [\frac 1 2, \frac 1 2 + \delta)$. Furthermore, if $\gamma\in J_p$, then the conditions Lemma~{\rm \ref{lem:LR}(c)} and \eqref{eq:p,gamma,a} hold for sufficiently small $a > 0$, where an upper bound for $a$ can be taken to be a continuous function of $\gamma$, which does not depend on $p$. By \eqref{eq:dim<}, we have
\[
\dim_H \mu \leq \frac{p\log p + (1-p)\log (1-p)}{\Big(1 - \frac{(1+\gamma)p^2(p+\gamma)}{\gamma - p(1-p)}\Big) \log a} < 1
\]
for $p \in [\frac 1 2, \frac 1 2 + \delta$), $\gamma\in J_{p}$ and sufficiently small $a > 0$ (with a bound depending continuously on $\gamma$).
In fact, analysing the inequalities \eqref{eq:gamma'} \eqref{eq:p,gamma}, Lemma~\ref{lem:LR}(c) and \eqref{eq:p,gamma,a}, one can obtain concrete ranges of  parameters $a, \gamma, p$, for which $\dim_H \mu < 1$ (cf.~Remark~\ref{rem:range} and Figure~\ref{fig:p-gamma}). 
\end{proof}

To complete the proof of Theorem~\ref{thm:singularity_lyap}, it remains to prove Lemma~{\rm \ref{eq:middle_below}}.

\begin{proof}[Proof of Lemma \rm \ref{eq:middle_below}]
		The proof is based on Kac's lemma (see Theorem~\ref{thm:kac}) and the observation that outside of the interval $M$, the system $\{f_-, f_+\}$ (after a logarithmic change of coordinates) acts like a random walk with a drift. Note first that $\mu(M) > 0$. Indeed, we have
	\begin{equation}\label{eq:f+x+}
	f_+^{-1}(x_-) > x_+,
	\end{equation}
	as it is straightforward to check that this inequality is equivalent to $a^{1 - \gamma} > 1$, which holds since $a \in (0,1)$ and $\gamma > 1$. This means that the sets $M$ and $f_+^{-1}(M)$ are not disjoint. By symmetry, $M$ and $f_-^{-1}(M)$ are also not disjoint. As $\lim \limits_{n \to 
		\infty} f_+^{-n}(x_-) = 0$ and $\lim \limits_{n \to \infty} f_-^{-n}(x_+) = 1$, we see that $\bigcup \limits _{n = 0}^{\infty} f_+^{-n}(M) \cup f_-^{-n}(M) = (0,1)$ and hence $\mu(M)>0$, as $\mu$ is stationary and $\mu(\{0,1\})=0$.
	
	We will apply Kac's lemma to the step skew product \eqref{eq:skew} and the set $\Sigma_2^+ \times M$. Let $n_{M} \colon \Sigma_2^+ \times M \to \N \cup \{ \infty \}$ be the first return time to $\Sigma_2^+ \times M$, i.e.
	\[ n_{M} (\ii, x) = \inf \{ n \geq 1 : (\FF^+)^n(\ii, x) \in \Sigma_2^+ \times M \}. \]
Set $\PP = \Ber^+_{p_-, p_+}$ to be the $(p_-, p_+)$-Bernoulli measure on $\Sigma_2^+$. Since $\PP \otimes \mu$ is invariant and ergodic for $\FF^+$ (cf. \cite[Lemmas 3.2 and A.2]{homburg}) and $(\PP \otimes \mu) (\Sigma_2^+ \times M) = \mu(M)>0$, Kac's lemma  implies 
	\begin{equation}\label{eq:kac}
	\int \limits_{\Sigma_2^+ \times M} n_{M}\, d\nu = \frac{1}{\mu(M)},
	\end{equation}
	where 
	\[\nu = \frac{1}{\mu(M)} (\PP \otimes \mu)|_{\Sigma_2^+ \times M}.\]
	Recall that we assume the condition Lemma~\ref{lem:LR}(c), so $\overline L \cap \overline R = \emptyset$. Let 
	\[
	C = [\sup L, \inf R],
	\]
	so that $M = L \cup C \cup R$ with the union being disjoint. By the definitions of $L, C$ and $R$,
	\begin{equation}\label{eq:exit} f_-(L) \subset [0, x_+), \qquad f_-(C \cup R) \cup f_+(L \cup C) \subset M, \qquad f_+(R) \subset (x_-, 1].
	\end{equation}
	Let \[E = \{ (\ii, x) \in \Sigma_2^+ \times M : f_{i_1}(x) \notin M \} = \{ (\ii, x) \in \Sigma_2^+ \times M : n_{M} > 1 \}.\]
	It follows from \eqref{eq:exit} that
	\begin{equation}\label{eq:E}
	E = \{i_1 = -\} \times L \cup \{ i_1 = + \} \times R,
	\end{equation}
	so
	\begin{equation}\label{eq:E_measure}
	\nu(E) = \frac{p_- \:\mu(L) + p_+ \:\mu(R)}{\mu(M)}
	\end{equation}
	and as $L$, $R$ are disjoint subsets of $M$,
	\begin{equation}\label{eq:half}
	\nu(E) \le p\frac{\mu(L) + \mu(R)}{\mu(M)} \leq p
	\end{equation}
for $p = \max(p_-, p_+)$.
	By \eqref{eq:E}, 
	\begin{equation}\label{eq:return_int1} 
		\int \limits_{\Sigma_2^+ \times M} n_{M} \, d\nu = 1 - \nu (E)  +  \int \limits_E n_{M} \, d\nu = 1 - \nu (E)  +  \int \limits_{\{i_1 = -\} \times L} n_{M} \, d\nu +  \int \limits_{\{i_1 = +\} \times R} n_{M} \, d\nu.
	\end{equation}
		Note that it follows from \eqref{eq:f+x+} that $f_+(x_+) < x_-$, hence a trajectory $\{f_{i_n} \circ \cdots \circ f_{i_1}(x)\}_{n=0}^\infty$ of a point $x \in [0,1]$ cannot jump from $[0, x_+)$ to $(x_-, 1]$ (or vice versa) without passing through $M$. Combining this observation with the fact that the transformations $f_-$ and $f_+$ are increasing, we conclude that 
	\begin{equation}\label{eq:n}
	\begin{aligned}
	n_{M}(\ii, x) &\leq n_{M}(\ii, x_+) \quad \text{for }(\ii, x) \in \{i_1 = -\} \times L,\\
	n_{M}(\ii, x) &\leq n_{M}(\ii, x_-) \quad \text{for }  (\ii, x) \in \{i_1 = +\} \times R.
	\end{aligned}
	\end{equation}
Therefore, we can apply \eqref{eq:n} together with \eqref{eq:E_measure} to obtain
	\begin{equation}\label{eq:return_int2} 
	\begin{aligned}
	&\int \limits_{\{i_1 = -\} \times L} n_{M} \, d\nu +  \int \limits_{\{i_1 = +\} \times R} n_{M} \, d\nu \leq  \int\limits_{\{i_1 = -\} \times L} n_{M}(\ii, x_+)d\nu + \int\limits_{\{i_1 = +\} \times R} n_{M}(\ii, x_-)d\nu\\
	& = \frac{\mu(L)}{\mu(M)} \int\limits_{\{i_1 = -\}} n_{M}(\ii, x_+)d\PP(\ii) + \frac{\mu(R)}{\mu(M)} \int\limits_{\{i_1 = +\}} n_{M}(\ii, x_-)d\PP(\ii)\\
	& = p_-\:\frac{\mu(L)}{\mu(M)}\mathbb{E}_-N_- + p_+\:\frac{\mu(R)}{\mu(M)}\mathbb{E}_+N_+\le\nu(E) \max(\mathbb{E}_-N_-, \mathbb{E}_+N_+),
	\end{aligned}
	\end{equation}
where 
\[
N_\pm(\ii) = \inf \{ n \geq 1 : f_{i_n} \circ \ldots \circ f_{i_1} (x_\mp) \in M \}
\]
and $\mathbb{E}_\pm$ is the expectation taken with respect to the conditional measure \[
\PP_\pm = \frac{1}{\PP(i_1 = \pm)} \PP |_{\{i_1 = \pm\}} = \frac{1}{p_\pm} \PP |_{\{i_1 = \pm\}}. 
\]
Using \eqref{eq:return_int1}, \eqref{eq:return_int2} and \eqref{eq:half}, we obtain
\begin{equation}\label{eq:return_int} 
\int \limits_{\Sigma_2^+ \times M} n_{M} \, d\nu \le1 + \nu(E) (\max(\mathbb{E}_-N_-, \mathbb{E}_+N_+) - 1)\le 1 + p(\max(\mathbb{E}_-N_-, \mathbb{E}_+N_+)-1).
\end{equation}

Define random variables $X_j^\pm \colon \Sigma_2^+ \to \R$, $j \in\N$, by
	\[ 
	X_j^-(\ii) = \begin{cases} 1 &\text{if } i_j = - \\ -\gamma &\text{if } i_j = + \end{cases}, \qquad X_j^+(\ii) = \begin{cases} -\gamma &\text{if } i_j = - \\
	1 &\text{if } i_j = + \end{cases}
	.\]
	Then $X_2^-, X_3^-, \ldots$ is an i.i.d.~sequence of random variables with $\PP_-(X_j^- = 1) = p_-$, $\PP_-(X_j^- = -\gamma) = p_+$. To estimate $\mathbb{E}_-N_-$, note that for $\ii \in \{i_1 = -\}$ we have
	\[ N_-(\ii) = \inf \{ n \geq 1 : a^{1 + X_2^- + \ldots + X_n^-} x_+ \geq x_+\} = \inf \{ n \geq 2 : X_2^- + \ldots + X_n^- \leq -1\},  \]
	as for $n < N_-(\ii)$ we have $f_{i_n} \circ \ldots \circ f_{i_1} (x_+) < x_+$ and $f_-(x) = ax,\ f_+(x) = a^{-\gamma}x$ on $[0,x_+]$. Consequently, $N_-$ is a stopping time for $\{X_j^-\}_{j=2}^{\infty}$. We show that $\mathbb{E}_-N_- < \infty$. To do this, note that by Hoeffding's inequality (see Theorem~\ref{thm:hoef}) and \eqref{eq:gamma},
	\begin{align*} \PP_-(N_- > n + 1) &\leq \PP_-\Big(\sum \limits_{j=2}^{n+1} X_j^- > -1\Big)\\ &= \PP_-\Big(\sum \limits_{j=2}^{n+1} X_j^- - n\mathbb{E}_- X_2^- \geq -1 - n(p_- - \gamma p_+)\Big)\\ &\leq \exp \Big(- \frac{2(1 + n(p_- - \gamma p_+))^2}{n (\gamma + 1)^2} \Big) \leq \exp(-cn)
	\end{align*}
	for some constant $c>0$ and $n \in \N$ large enough. We have used here the fact that $t:= - 1 - n(p_- - \gamma p_+)$ is positive for $n$ large enough, following from  \eqref{eq:gamma}. As $\mathbb{E}_-N_- = \sum \limits_{n=0}^\infty \PP_-(N > n)$, the above inequality implies $\mathbb{E}_-N_- < \infty$. 
	
Let 
\[ 
S_{N_-}(\ii) = \sum \limits_{n=2}^{N_-(\ii)} X_n^-(\ii).
\]
This random variable is well-defined, since $2 \leq N_- < \infty$ holds $\PP_-$-almost surely. As $\mathbb{E}_-N_- < \infty$, we can apply Wald's identity (see Theorem~\ref{thm:wald}) to obtain
	\begin{equation}\label{eq:wald} \mathbb{E}_-S_{N_-} = \mathbb{E}_- X_2^-(\mathbb{E}_-N_- -1) = (p_- - \gamma p_+) (\mathbb{E}_-N_- -1).
	\end{equation}
	In order to estimate  $\mathbb{E}_-S_{N_-}$, we condition on $X_2^-$ and note that $S_{N_-} \geq -1-\gamma$ almost surely and, by \eqref{eq:gamma},  $-\gamma < -1$. This gives
	\begin{equation}\label{eq:ESN bound}
	\begin{split}
	\mathbb{E}_- S_{N_-} & = p_-\: \mathbb{E}_-(S_{N_-}|X_2^- = 1)  + p_+ \: \mathbb{E}_-(S_{N_-}|X_2^- = -\gamma) \\
	& \geq p_- \: (-1 - \gamma) - p_+\: \gamma \\
	& = -p_- - \gamma.
	\end{split}
	\end{equation}
	Combining this with \eqref{eq:gamma} and \eqref{eq:wald} we get
	\begin{equation}\label{eq:EN-}
	\mathbb{E}_- N_- -1 \leq \frac{p_- + \gamma}{\gamma p_+ - p_-}. 
	\end{equation}
	
By the symmetry \eqref{eq:sym} and $x_+ = \II(x_-)$, we can estimate $\mathbb{E}_+N_+$ in the same way, exchanging the roles of $p_-$ and $p_+$, obtaining
\begin{equation}\label{eq:EN+}
\mathbb{E}_+ N_+ -1 \leq \frac{p_+ + \gamma}{\gamma p_- - p_+}.
\end{equation}
Applying \eqref{eq:EN-} and \eqref{eq:EN+} to \eqref{eq:return_int} we see that
	\[ \int \limits_{\Sigma_2^+ \times M} n_{M}(\ii, x)d\nu (\ii, x) \leq 1 + p \max\Big(\frac{p_- + \gamma}{\gamma p_+ - p_-}, \frac{p_+ + \gamma}{\gamma p_- - p_+} \Big) =  1 + p \frac{p + \gamma}{\gamma (1-p) - p}. \]
	Invoking \eqref{eq:kac}, we obtain
	\[
	\mu(M) \ge \frac{1}{1 + p \frac{p + \gamma}{\gamma (1-p) - p}} = \frac{\gamma (1-p)-p}{\gamma - p(1-p)},
	\]
which ends the proof.
\end{proof}

We finish the paper with some remarks on the limitations of our method for proving singularity of the measure $\mu$.

\begin{rem} One should be aware that in general, the upper bound $-\frac{H(p_-, p_+)}{\chi(\mu)}$ does not coincide with the actual value of $\dim_H \mu$ for AM-systems. Indeed, for $(p_-, p_+) = (\frac 1 2, \frac 1 2)$ we have $H(p_-, p_+) = \log 2$ and by \eqref{eq:chi formula}
	\[
	\chi(\mu) = \Big( \frac{1+ \gamma}{2} \mu(M) + \frac{1 - \gamma}{2} \Big) \log a \ge  \log a.
	\]
	On the other hand, \cite[Theorems 2.10 and 2.12]{BS19Singular} provide an exact value of the dimension of $\mu$ in the resonance case $\gamma = k \in \N$, $k \ge 2$, yielding
	\[ \hdim\mu = \hdim(\supp \mu) = \frac{\log \eta}{\log a}, \]
	where $\eta \in (\frac 1 2, 1)$ is the unique solution of the equation $\eta^{k+1} - 2\eta + 1 = 0$. Therefore,
	\[ \hdim\mu = \frac{\log \eta}{\log a} < \frac{\log \frac{1}{2}}{\log a} \leq  -\frac{H(p_-, p_+)}{\chi(\mu)}.\]
\end{rem}

\begin{rem} It is natural to ask, what is a possible range of parameters, for which the method presented is this paper could be applied. Let us discuss this in the basic case $p_- = p_+ = \frac 1 2$. Following the proof of Theorem~\ref{thm:singularity_lyap} in this case, we see that by \eqref{eq:entropy_lyap} and \eqref{eq:chi formula},  if for a given $\gamma > 1$ we have
\[\mu(M) > \frac{\gamma - 1}{\gamma + 1},\]
then the measure $\mu$ is singular for $a < 1$ small enough (depending on $\gamma$). On the other hand, combining \eqref{eq:kac}, \eqref{eq:return_int}, \eqref{eq:wald} and noting that $\mathbb{E}_- N_- = \mathbb{E}_+ N_+$ and $\mathbb{E}_- S_{N_-} = \mathbb{E}_+ S_{N_+}$ for $p_- = p_+ = \frac 1 2$, we see that
\[ \mu(M) \geq \frac{\gamma - 1}{\gamma - 1 - \mathbb{E}_- S_{N_-}}, \]
provided that the condition of Lemma~\ref{lem:LR}(c) is satisfied (which for a fixed $\gamma > 1$ holds for small enough $a \in (0,1)$). Therefore, if for fixed $\gamma > 1$ inequality
\begin{equation}\label{eq:ESN > -2} \mathbb{E}_- S_{N_-} > -2
\end{equation}
is satisfied, then $\mu$ is singular for $a \in (0,1)$ small enough. The proof of Theorem \ref{thm:singularity_lyap} shows that \eqref{eq:ESN > -2} holds for $\gamma \in (1, \frac 3 2)$. 
Figure~\ref{fig:ESN} presents computer simulated values of $\mathbb{E}_- S_{N_-}$ for $\gamma$ in the interval $(1,3)$.
\begin{figure}[ht!]
	\begin{center}
		\includegraphics[width=0.9\textwidth]{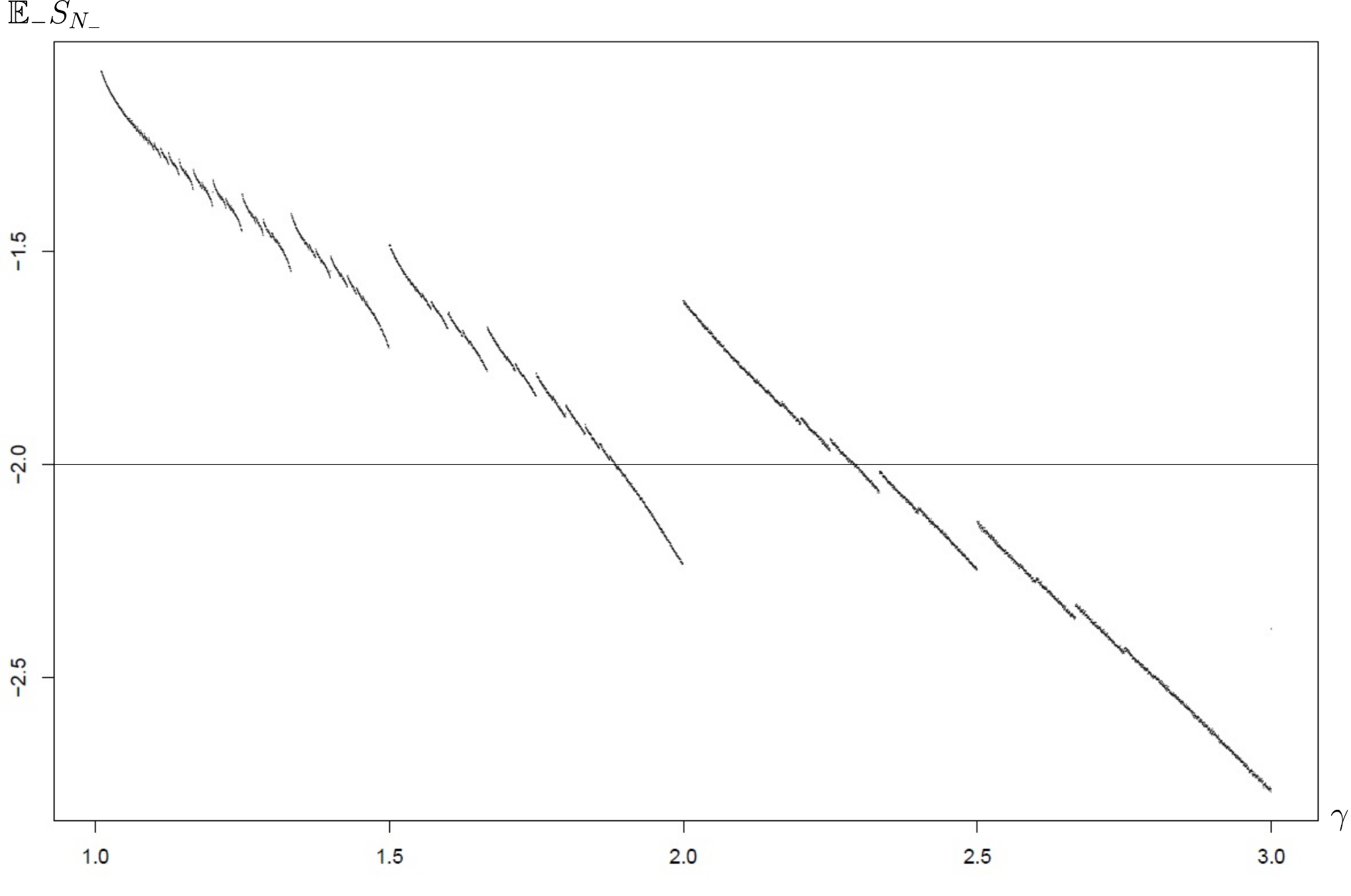}
	\end{center}
	\caption{Simulated values of $\mathbb{E}_- S_{N_-}$ as a function of $\gamma$. The values of $\gamma$ are presented on the $x$-coordinate axis, while the $y$-coordinate gives the corresponding value of $\mathbb{E}_- S_{N_-}$. Simulations were performed for $4000$ values of $\gamma$, uniformly spaced in the interval $(1,3)$. For each choice of  $\gamma$, we performed $40000$ simulations of $3000$ steps of the corresponding random walk.}\label{fig:ESN}
\end{figure} 
It suggests that the range of parameters $\gamma$ for which the singularity of $\mu$ holds with $a$ small enough could be extended from $(1, \frac{3}{2})$ to a larger set of the form $(1, \gamma_1) \cup (2, \gamma_2)$, for some $\gamma_1 \in (1,2)$, $\gamma_2 \in (2, 3)$. It is easy to see that one can obtain \eqref{eq:ESN > -2} for some $\gamma>\frac 3 2$ by conditioning on a larger number of steps in \eqref{eq:ESN bound}. We do not pursue the task of finding a wider set of possible parameters in this work. One should note, however, that \eqref{eq:ESN > -2} cannot hold for $\gamma \geq 3$, as the formula from the first line of \eqref{eq:ESN bound} can be used together with an obvious bound $S_{N_-} \leq -1$ to obtain (for $p_- = p_+ = \frac 1 2$)
\[ \mathbb{E}_- S_{N_-} \leq -p_- - p_+ \gamma = \frac{- 1 - \gamma}{2},\]
yielding $\mathbb{E}_- S_{N_-} \leq -2$ for $\gamma \geq 3$. This shows that the method used in this paper cannot be (directly) applied for $\gamma \geq 3$ (even though there do exist AM-systems with $\gamma \geq 3$ for which $\mu$ is singular -- see \cite[Theorems 2.10 and 2.12]{BS19Singular}). In order to obtain an optimal range of $\gamma$'s satisfying \eqref{eq:ESN > -2}, one should compute $\mathbb{E}_- S_{N_-}$ explicitly in terms of $\gamma$. This however seems to be complicated and Figure~\ref{fig:ESN} suggests that one should not expect a simple analytic formula.
\end{rem}

\addcontentsline{toc}{chapter}{Bibliography}
\bibliographystyle{plain}
\bibliography{universal_bib}

\end{document}